\documentclass{amsart}
\usepackage[utf8]{inputenc}
\usepackage[english]{babel}
\usepackage[makeroom]{cancel}
\usepackage{amsmath}
\usepackage{amsfonts}
\usepackage{amssymb}
\usepackage{mathtools}
\usepackage{pgf,tikz,pgfplots} 
\pgfplotsset{compat=1.15}
\usepackage{graphicx}
\usepackage{float}
\usepackage{mathrsfs}
\usetikzlibrary{arrows.meta,positioning,automata,shadows}
\usepackage{colortbl}
\usepackage{subfigure}
\usepackage{hyperref}

\usepackage[most]{tcolorbox}
\newcommand{\algo}[2]{
    \begin{tcolorbox}[breakable,colback=black!5!white,colframe=black,title={#1}]
        #2
    \end{tcolorbox}
}

\newtheorem{thm}{Theorem}[section]
\newtheorem*{thm*}{Theorem}
\newtheorem{lemme}[thm]{Lemma}

\newtheorem{prop}[thm]{Proposition}

\theoremstyle{definition}
\newtheorem{defi}[thm]{Definition}

\theoremstyle{remark}
\newtheorem{nota}[thm]{Notation}
\newtheorem{conj}[thm]{Conjecture}

\newtheorem{rem}[thm]{Remark}
\newtheorem{ex}[thm]{Example}

\newcommand{\abs}[1]{\left\lvert#1\right\rvert}

\newcommand{\acts}{\curvearrowright}

\DeclareMathOperator{\Stab}{Stab}
\DeclareMathOperator{\ev}{ev}
\DeclareMathOperator{\GL}{GL}
\DeclareMathOperator{\PGL}{PGL}

\DeclareMathOperator{\Tr}{Tr}
\DeclareMathOperator{\Orb}{Orb}

\DeclareMathOperator{\PSL}{PSL}

\DeclareMathOperator{\e}{e}
\DeclareMathOperator{\SL}{SL}

\newcommand{\RR}{\mathbb{R}}
\newcommand{\CC}{\mathbb{C}}
\newcommand{\NN}{\mathbb{N}}

\newcommand{\ZZ}{\mathbb{Z}}
\newcommand{\QQ}{\mathbb{Q}}
\newcommand{\PP}{\mathbb{P}}

\newcommand{\cO}{\mathcal{O}}
\newcommand{\cQ}{\mathcal{Q}}
\newcommand{\cN}{\mathcal{N}}

\begin{document}

\title{Burau representation of $B_4$ and quantization of the rational projective plane}
\author{Perrine Jouteur}
\date{}

\address{
Perrine Jouteur,
Laboratoire de Math\'ematiques de Reims, UMR~9008 CNRS et
Universit\'e de Reims Champagne-Ardenne,
U.F.R. Sciences Exactes et Naturelles,
Moulin de la Housse - BP 1039,
51687 Reims cedex 2,
France} 
\email{perrine.jouteur@univ-reims.fr}

\begin{abstract}
The braid group $B_4$ naturally acts on the rational projective plane $\PP^2(\QQ)$, this action corresponds to
the classical integral reduced Burau representation of~$B_4$.
The first result of this paper is a classification of the orbits of this action.
The Burau representation then defines an action of~$B_4$ on $\PP^2(\ZZ(q))$, where $q$ is a formal parameter
and $\ZZ(q)$ is the field of rational functions in~$q$ with integer coefficients.
We study orbits of the $B_4$-action on $\PP^2(\ZZ(q))$, and show existence of embeddings of
the $q$-deformed projective line~$\PP^1(\ZZ(q))$ that precisely correspond to the
notion of $q$-rationals due to Morier-Genoud and Ovsienko.
\end{abstract}

\maketitle
\vspace{-0.5 cm}
\tableofcontents

\newpage

\section{Introduction and main results}
  
The $4$-strands Artin braid group $B_4$ is generated by three elements $\sigma_1, \sigma_2, \sigma_3$ with braid relations
$$
\sigma_1\sigma_{2}\sigma_1=\sigma_{2}\sigma_{1}\sigma_{2}, \quad \sigma_{2}\sigma_{3}\sigma_{2}=\sigma_3\sigma_{2}\sigma_3,
$$
and commutation relation $\sigma_1\sigma_3=\sigma_3 \sigma_1$.

The classical reduced \textit{Burau representation} of $B_4$ is a group homomorphism
$$
\rho_q : B_4 \rightarrow \GL_3(\Lambda),
$$
where $ \Lambda := \ZZ[q,q^{-1}]$ is the ring of Laurent polynomials in one (formal) variable $q$ with integer coefficients, 
defined by
\begin{equation}
\label{Burau}
\rho_q(\sigma_1) = \begin{pmatrix}
							q & 1 & 0 \\
							0 & 1 & 0 \\
							0 & 0 & 1\\
							\end{pmatrix}, 
\qquad
\rho_q(\sigma_2) = \begin{pmatrix}
							1 & 0 & 0 \\
							-q & q & 1 \\
							0 & 0 & 1\\
							\end{pmatrix}, 
\qquad
\rho_q(\sigma_3) = \begin{pmatrix}
											1 & 0 & 0 \\
											0 & 1 & 0 \\
											0 & -q & q\\
											\end{pmatrix}.
\end{equation}											
Note that, for the sake of convenience and following \cite{Dlugie2022,morier-genoud_burau_2024} we have chosen the parameter $q=-t$, where~$t$ is
a more standard choice of parameter used in the theory of braid groups. 

The Burau representation goes back to Werner Burau~\cite{Burau} who 
used it to interpret the Alexander polynomial  of knots in algebraic terms.
Faithfulness of the representation~\eqref{Burau} is a long standing open problem.
For more details about the Burau representation, see~\cite{Birman,kassel_braid_2008}.

The main goal of this paper is to study the natural projective version of the Burau representation,
which is the action of $B_4$ on the projective plane $\PP^2(\bar{\Lambda})$
with coefficients in the field $\bar{\Lambda} := \ZZ(q)$.
Recall that the field $\bar{\Lambda}$ is the same as the field $\QQ(q)$ of rational functions in~$q$ and every $F(q)\in\bar{\Lambda}$ can be written in the form
$$
F(q)=\frac{R(q)}{S(q)},
$$
where $R$ and $S$ are polynomials in $q$ with integer coefficients.
The action of $B_4$ on $\PP^2(\overline{\Lambda})$ is defined as the projectivization of~\eqref{Burau}. We will still denote by $\rho_q$ this projective version of the Burau representation
$$\rho_q : B_4 \longrightarrow \PGL_3(\Lambda).$$
We understand this action as $q$-deformation, or ``quantization'' of the rational projective plane.
Our approach is similar to that of~\cite{MGO} where the case of the projective line was investigated.

\subsection{The case $q=1$, classification of orbits}

In the special case $q=1$, the homomorphims~\eqref{Burau} is the \textit{integral} Burau representation, 
$$
\rho: B_4 \rightarrow \SL(3,\ZZ),
$$
defined for the generators by
$$
\label{Burau1}
\rho(\sigma_1) = \begin{pmatrix}
							1 & 1 & 0 \\
							0 & 1 & 0 \\
							0 & 0 & 1\\
							\end{pmatrix}, 
\qquad
\rho(\sigma_2) = \begin{pmatrix}
							\,1 & 0 & 0 \\
							-1 & 1 & 1 \\
							\,0 & 0 & 1\\
							\end{pmatrix}, 
\qquad
\rho(\sigma_3) = \begin{pmatrix}
											1 &\, 0 & 0 \\
											0 &\, 1 & 0 \\
											0 & -1 & 1\\
											\end{pmatrix}.
$$										
Note that, unlike the Burau representation $\rho_q$ for which the question is wide open, 
it is known that the integral representation~$\rho$ has a nontrivial kernel.
The kernel of $\rho$  is a normal subgroup of $B_4$, called a \textit{braid Torelli group} and denoted by $\mathcal{B}\mathcal{I}_4$. 
Smythe in \cite{Smythe1979} found a set of normal generators of $\mathcal{B}\mathcal{I}_4$, i.e. a set of elements whose normal closure is $\mathcal{B}\mathcal{I}_4$. Smythe's set of normal generators of $\mathcal{B}\mathcal{I}_4$ is $\{ \tau_1^2, \tau_3^2, \Delta^2\}$, where
\begin{equation}
\label{GenKer}
\tau_1 = (\sigma_1\sigma_2\sigma_1)^2 ~\text{ , } \tau_3 = (\sigma_3\sigma_2\sigma_3)^2 \text{ and } \Delta = \sigma_1\sigma_2\sigma_3\sigma_1\sigma_2\sigma_1. 
\end{equation}

\noindent Note that $\Delta$ is the \textit{Garside element}, whose square $\Delta^2=(\sigma_1\sigma_2\sigma_3)^4 $ generates the center of $B_4$.

More recently Brendle, Margalit, and Putman gave a topological description of normal generators of $\mathcal{B}\mathcal{I}_n$ for all $n$, in \cite{Brendle_2014}.

The identification of the kernels of specializations of Burau representation at roots of unity was raised in \cite{Squier_1984} 
and developped in \cite{Dlugie2022}.

The projective version of $\rho$ gives rise to an action of $B_4$ on the rational projective plane $\PP^2(\QQ)$ that by slightly abusing the notation we also denote by $\rho$
$$
\rho : B_4 \rightarrow\PSL_3(\ZZ) \acts \PP^2(\QQ).
$$ 
Recall that the group $\PSL_3(\ZZ)$ actually coincides with $\SL_3(\ZZ)$.

Let us describe the action of $B_4$ on $\PP^2(\QQ)$ more explicitly.
Every point $p\in\PP^2(\QQ)$ has integer homogeneous coordinates:
$$
p=[r:s:t],
$$
where $r,s,t\in\ZZ$ are mutually prime. 
Every point has exactly two such representatives, that differ by the sign. 
The $\SL_3(\ZZ)$-action preserves this convention, and so does the $B_4$-action we are interested in. 
The $B_4$-action on $\PP^2(\QQ)$ is then
\begin{equation} 
\label{eq:1}
\begin{array}{rccl}
\rho(\sigma_1): & [r:s:t] & \mapsto & [r+s : s: t],\\[4pt]
\rho(\sigma_2): & [r:s:t] & \mapsto & [r:s+t-r:t],\\[4pt]
\rho(\sigma_3): & [r:s:t] & \mapsto & [r:s:t-s].
\end{array}
\end{equation}
While the group $\SL_3(\ZZ)$ acts transitively on $\PP^2(\QQ)$, 
the action of the braid group $B_4$ does not have this transitivity property. 
For example, the point $[1:0:1]$ is fixed by $\rho$ and constitutes the only orbit consisting in one point. 

Our first main result is a complete description of the orbits in $\PP^2(\QQ)$ for the $B_4$-action~\eqref{eq:1}.
Despite the fact that the question has a classic nature, we did not find this statement in the literature.

\begin{thm} 
(i)
\label{decomposition_orbits}
Under the $B_4$-action,
the rational projective plane is decomposed into infinitely many orbits as follows
$$
\PP^2(\QQ) = \{[1 : 0 : 1]\} ~\sqcup ~
\Orb_{B_4}([0:1:0]) ~\sqcup ~
\bigsqcup_{n \geq 2, 0 < m < n/2 \atop m\wedge n = 1} \Orb_{B_4}([m:n:m]).
$$

(ii)
For every couple $(m,n)$ of coprime integers, the orbit  $\Orb_{B_4}([m:n:m])$ consists in the following points
$$
\Orb_{B_4}([m:n:m]) =
 \left\{[r:s:t] ~\Big\vert~ 
 \begin{cases}
 \gcd(r-t,s) = 
 n ;\\
 r,t\equiv \pm m \!\!\! \mod(n).
 \end{cases}
 \right\},
$$
\noindent and
$$
\Orb_{B_4}([0:1:0]) = \left\{[r:s:t] ~|~ \gcd(r-t,s) = 1 \right\}.
$$
\end{thm}

\noindent
This theorem will be proved in Sections~\ref{FirstProofSec} and~\ref{AlgoSec}.

Besides the singleton orbit $\{[1 : 0 : 1]\}$, every $B_4$-orbit contains infinitely many points.
Moreover, we will show that every such orbit is dense in~$\PP^2(\QQ)$.
However, the orbit of the point $[0:1:0]$ is (conjecturally)  the ``largest'' orbit in the following sense.
For every $N\in\NN$, 
the orbit $\Orb_{B_4}([0:1:0])$ contains at least three times more points
in the subset $\{[r:s:t]\,\vert\,|r|,|s|,|t|\leq N\}$ of the rational plane than the union of the other orbits.
Although we do not have a proof of this statement, we will give the numerical evidence for this ``experimental fact''.
We will refer to this orbit as the ``principal orbit'' and use the special notation
$$
\cO_1 := \Orb_{B_4}([0:1:0]).
$$

Let $\Stab_{[m:n:m]}\subset B_4$ be the stabilizer of a point $[m:n:m]$.
Clearly, $\mathcal{B}\mathcal{I}_4\subset\Stab_{[m:n:m]}$.
The next result gives a complete description of the stabilizers modulo the braid Torelli group $\mathcal{B}\mathcal{I}_4$.
Note in particular that the stabilizer $\Stab_{[0:1:0]}$ of the point $[0:1:0]$ is generated by $\sigma_2$, $\Delta$ and $\tau_1$. 

\begin{thm}
\label{StabConj}
Let $n\in \NN^*$ and $0 \leq m < n$ coprime with $n$. Then
$$
\Stab_{[m:n:m]}/\mathcal{B}\mathcal{I}_4 = 
\left\lbrace
\begin{array}{l l}
\langle ~ \tau_1\Delta,~ \sigma_2 ~\rangle & \text{ if } n\geq 3,\\[4pt]
\langle ~ \tau_1\Delta, ~\sigma_2, ~\sigma_1\sigma_2^2\sigma_3 ~\rangle & \text{ if } n = 2,\\[4pt]
\langle ~\tau_1, ~\Delta, ~\sigma_2 ~\rangle & \text{ if } n = 1.\\
\end{array}
\right.
$$
\end{thm}

\noindent
This statement will be proved in Section~\ref{StabSec}.

\subsection{Quantization procedure, comparison to $q$-rationals}
We  introduce the notion of
\textit{quantization of the rational projective plane} $\PP^2(\QQ)$.
The quantization map is a set-valued function
$$
\cQ:\PP^2(\QQ)\to \mathcal{P}\big(\PP^2(\overline{\Lambda})\big).
$$
It associates to every point $p = [r:s:t]$ of  $\PP^2(\QQ)$ an infinite set of points $[R(q):S(q):T(q)]$, 
called the quantization of $p$, where $R$, $S$ and $T$ are polynomials in~$q$ with integer coefficients.

The precise definition is as follows.
For every $p\in\PP^2(\QQ)$, in the orbit of $[m:n:m]$, we set
\begin{equation}
\label{QuantMap}
\cQ(p) := \left\{\rho_q(\beta)([m:n:m]) ~|~\beta \text{ s.t. } \rho(\beta)([m:n:m]) = p\right\}.
\end{equation}

We will be mostly interested in the quantization of the principal orbit $\cO_1$.
The image of $\cO_1$ with respect to the quantization map will be denoted by $\cO_q$.

The above quantization procedure is analogous to the notion of \textit{$q$-deformed rationals} introduced in~\cite{MGO}.
The main difference is that the image of one point by our quantization map~\eqref{QuantMap} consists in an infinite number of points.
We will explain this phenomenon is Section~\ref{XXX}.

Let us briefly describe the quantization procedure of Morier-Genoud and Ovsienko
using the terms which are most close to our context.
Consider the rational projective line $\PP^1(\QQ)$ equipped with the standard transitive action of the modular group $\PSL(2,\ZZ)$
$$
\begin{pmatrix}
a & b \\
c & d \\
\end{pmatrix}:
[r:s]\longmapsto[ar+bs:cr+ds].
$$
The $\PSL(2,\ZZ)$-action can also be considered as an action of the braid group $B_3$ (the center acts trivially).
The Burau representation $\rho_q$ of $B_3$ then defines a $\PSL(2,\ZZ)$-action on $\PP^1(\overline{\Lambda})$.
For the generators,
$$
\rho_q(\sigma_1):  [r:s]  \mapsto  [qr+s : s],
\qquad
\rho_q(\sigma_2): [r:s] \mapsto [r:q(s-r)].
$$
The quantization of Morier-Genoud and Ovsienko is the unique map 
$$
\cQ:\PP^1(\QQ)\to\PP^1(\overline{\Lambda})
$$
that commutes with the $\PSL(2,\ZZ)$-action and sends $[0:1]$ to $[0:1]$ (this point remains unchanged).
It associates to a point $[r:s]$ a pair of monic polynomials with positive integer coefficients $(R(q),S(q))$.
In other words, the $q$-rationals are defined as the orbit of the point $[0:1]\in\PP^1(\overline{\Lambda})$ under the Burau representation.

The notion of $q$-rationals enjoys a number of remarkable properties,
we mention only few of them:
\begin{itemize}
\item
the ``total positivity'' property~\cite{MGO} that means
roughly speaking that the topology of $\PP^1(\QQ)$ is preserved by quantization;

\item
the unimodality property, conjectured in~\cite{MGO} and eventually proved in~\cite{oguz_rank_2023},
asserts that the sequences of coefficients of the polynomials $R(q)$ and $S(q)$ are unimodal;

\item
a connection to the Jones polynomial of rational knots (see~\cite{MGO, Sikora_2023});

\item
the stabilization phenomenon that led to the notion of a $q$-deformed real number~\cite{qreals}.

\end{itemize}
\noindent
An application of $q$-rationals to the Burau representation of $B_3$ was suggested in~\cite{morier-genoud_burau_2024},
where the information about the polynomials arising in the process of quantization of $\PP^1(\QQ)$
was sufficient to guarantee faithfulness of specializations of the Burau representation.
This application is an important motivation for our study,
we hope that the polynomials $R,S,T$ will be useful for the study of the Burau representation of $B_4$.

\subsection{An embedding of the projective line}

Consider the following embedding of the projective line into the projective plane 
$$
\PP^1(\QQ)\overset{\iota}{\hookrightarrow}\PP^2(\QQ),
\qquad\qquad
[r:s]\mapsto[r:s:0].
$$ 
It is easy to see that its image belongs to the principal orbit $\cO_1$. 
Similarly, we define the embedding
$$
\PP^1(\overline{\Lambda})\overset{\iota_q}{\hookrightarrow}\PP^2(\overline{\Lambda}),
\qquad\qquad
[R(q):S(q)]\mapsto[R(q):S(q):0].
$$ 
We will prove that the above embeddings commute with quantization.

\begin{thm}
\label{QuantThm}
The quantization of the projective line in the sense of~\cite{MGO} 
and our quantization
commute with the embeddings. 
In other words, for every $[r:s] \in \PP^1(\QQ)$, we have
$$\iota_q(\cQ([r:s])) \in \cQ(\iota([r:s])).$$
\end{thm}

\noindent
This statement will be proved in Section~\ref{QuantThmSec}.

The other natural embedding of the projective line into the projective plane, namely
$[r:s]\mapsto[0:r:s]$ also commutes with quantization.

\section{The structure of orbits}

In this section, we prove Theorem~\ref{decomposition_orbits} and Theorem \ref{StabConj}, and we introduce the braided Euclidean algorithm.

\subsection{Proof of Theorem~\ref{decomposition_orbits}, first part}\label{FirstProofSec}
In this section, we prove Theorem~\ref{decomposition_orbits} in one direction.
We check that the orbits $\Orb_{B_4}([m:n:m])$ with different values of (coprime) $m$ and $n$ 
and $0 \leq m < n/2$,
are indeed disjoint. 

\begin{lemme} \label{invariant}
(i) 
For every $[r:s:t] \in \PP^2(\QQ)$, the number 
$$
n:=\gcd(r-t,s)
$$ 
is invariant under the action of $B_4$. 

(ii)
Up to the sign, the number
$
r \!\!\mod(n)=t \!\!\mod(n)
$ 
is invariant under the action of $B_4$.
\end{lemme}

\begin{proof}~\\
	\noindent (i). As $\gcd(r+s-t,s) = \gcd(r-t,s+r-t) = \gcd(r-t,s)$, the actions of the generators $\sigma_1$, $\sigma_2$, $\sigma_3$ given by~\eqref{eq:1}
do not change the quantity $\gcd(r-t,s)$.\\
~\\
\noindent (ii). The quantities $r \!\!\mod(n)$ and $t \!\!\mod(n)$ coincide because $n$ divides $r-t$ so $r\equiv t \mod(n)$. 
Because of the sign change $[r:s:t] = [-r:-s:-t]$, this value is only defined up to the sign.
It is straightforward to check that it is invariant under the action of the generators of $B_4$.
\end{proof}

The above lemma implies that if a point $[r:s:t]$ belongs to the orbit of $[m:n:m]$, then $\gcd(r-t,s) = n$ and $r$ and $s$ are congruent to $\pm m$ modulo $n$.
In particular we deduce that different points $[m:n:m]$, with $n \in \NN^*$ and $m$ coprime to $n$ such that $0 \leq m < n/2$, 
belong to different orbits.

\subsection{The condition $0 \leq m < n/2$, further examples of orbits}

Let us now explain why the condition $0 \leq m < n/2$ is necessary to have disjoint orbits.

Consider a point $[m:n:m]\in\PP^2(\QQ)$ with $m<n$.
One then has
$$
[m:n:m] = \rho(\sigma_1\sigma_2^2\sigma_3) \left([n-m:n:n-m]\right),
$$
so that $[m:n:m]$ and $[n-m:n:n-m]$ belong to the same orbit.
Therefore the condition $0 \leq m < n/2$ is indeed necessary to have different orbits.
Note also that when $n\geq2$, one cannot take $m=0$ since
$[0:n:0]=[0:1:0]$, so that the condition reads $0 < m < n/2$ in this case.

\begin{ex}
(i)
For $n = 2$, our list of orbits contains only one orbit, which is the orbit of $[1:2:1]$. 
More explicitly, the orbit $\Orb_{B_4}([1:2:1])$ consists in the points
$[r:s:t]$ such that $r$ and $t$ are odd, $s$ is even and $ \gcd(r-t,s) = 2$.

(ii) 
For $n= 3$, there is also only one orbit, the orbit of $[1:3:1]$. 
The point $[2:3:2]$ is recovered for instance by 
$$
[1:3:1] \overset{\sigma_3}{\longrightarrow} [1:3:-2] 
\overset{\sigma_2^{2}}{\longrightarrow} [1:-3:-2] = [-1:3:2] 
\overset{\sigma_1}{\longrightarrow} [2:3:2].
$$
The orbit $\Orb_{B_4}([1:3:1])$ consists of the points $[r:s:t]$ such that $s$ is a multiple of $3$ and $ \gcd(r-t,s) = 3$.
\end{ex}

\subsection{An Euclid-like algorithm, end of the proof of Theorem~\ref{decomposition_orbits}}\label{AlgoSec}
In this subsection, we finish the proof of Theorem~\ref{decomposition_orbits}.
We show that every point of $\PP^2(\QQ)$ belongs to the orbit of $[m:n:m]$ for some $m$ and $n$. 
For this end, we construct an explicit way to go from a point $[r:s:t]$, to the corresponding representative $[m:n:m]$. 
Note that this algorithm fits into the framework of Jacobi-Perron type multicontinued fraction algorithm as described in \cite{rada_periodicity_2024}. 

\vspace{2mm}

\algo{Braided Euclidean algorithm}{
{\noindent{\bf \underline{Input}:}}
We start with a point $[r:s:t] \neq [1:0:1]$, with $r,s,t \in \ZZ$, mutually prime. We can assume that $s \geq 0$.

\vspace{1mm}

{\noindent {\bf \underline{Step 1 of the algorithm}:}}
If $s = 0$, then apply $\sigma_2$ to replace $s$ by $t-r$, and if necessary change signs to get $s > 0$.\\
\noindent Write the Euclidean division of $r$ by $s$ and $t$ by $s$ :
$$r = sa_1 + r' \text{ ~,~ } t = sc_1 + t'.$$
\noindent Apply $\sigma_1^{-a_1}\sigma_3^{c_1}$, so that $[r:s:t] \mapsto [r':s:t'] =: [r_1:s_1:t_1]$.\\

\vspace{1mm}

{\noindent {\bf \underline{Step 2 of the algorithm}:}}
While $r_i-t_i \neq 0$, repeat :\\
\noindent Write the upper Euclidean division of $s_i$ by $(r_i-t_i)$, i.e. 
$$s_i = (r_i-t_i)b_{2i} + s' \text{ with } 0 < s' \leq \abs{r_i-t_i}.$$ 
\noindent Apply $\sigma_2^{b_{2i}}$ and put $s_{i+1} := s'$.\\
\noindent Write the Euclidean divisions of $r_i$ and $t_i$ by $s_{i+1}$ : 
$$r_i = s_{i+1}a_{2i+1} + r' \text{ ~,~ } t_i = s_{i+1}c_{2i+1} + t'.$$
\noindent Apply $\sigma_1^{-a_{2i+1}}\sigma_3^{c_{2i+1}}$ and put $r_{i+1} := r'$ and $t_{i+1} := t'$.

\vspace{1mm}

{\noindent{\bf \underline{Termination of the algorithm}:}}
If we have arrived to $r_i = t_i$, we do not proceed further. Here the algorithm terminates.
}

\vspace{1mm}

\begin{ex}
Let us apply the algorithm to the point $x = [37:30:12]$. As $\gcd(37-12,30) = 5$ and $37 \equiv 2 \mod(5)$, we know that $x$ is in the orbit of $[2:5:2]$. \\
\noindent First step : $37 = 30*1+7$ so $x  \overset{\sigma_1^{-1}}{\longrightarrow} [7:30:12]$.\\
\noindent Second step : $30 = \abs{7-12}*5 + 5$ so $[7:30:12] \overset{\sigma_2^{-5}}{\longrightarrow} [7:5:12]$ and then $[7:5:12] \overset{\sigma_1^{-1}\sigma_3^{2}}{\longrightarrow} [2:5:2]$.\\
\noindent Finally the braid $\sigma_1^{-1}\sigma_3^{2}\sigma_2^{-5}\sigma_1^{-1}$ sends $x$ to the representative of its orbit, $[2:5:2]$.

\end{ex}

\begin{prop} ~\\
\label{Terminator}
(i) The braided Euclidean algorithm described above terminates.\\
~\\
(ii) Given a point $p = [r:s:t] \in \PP^2(\QQ)$, with $n = \gcd(r-t,s)$ and $m = r \mod(n)$, the braided Euclidean algorithm provides a braid 
$\beta_p \in B_4$ such that
$$
\rho(\beta_p) \left(p \right) = [m:n:m],
$$ 
\end{prop}

\begin{proof}
At each step of the algorithm, by the Euclidean division property, the inequalities below hold:
$$ (\text{I}) ~\abs{r_{i+1}-t_{i+1}} < s_{i+1} < \abs{r_i-t_i} < s_i \text{ ~and~ } (\text{II})~ 0 \leq r_i < s_i,~~0\leq t_i <s_{i}. $$

\noindent Therefore the sequences $(s_i)_i$ and $(\abs{r_i-t_i})_i$ are strictly decreasing. The sequence $(\abs{r_i-t_i})_i$ reaches $0$ in a finite number $N$ of steps, so the algorithm terminates.\\
~\\
\noindent Moreover, when $i = N$, we have $r_N = t_N$ so in the end we get a point of the type $[m:n:m]$, with $0 \leq m < n$ by $(\text{II})$. \\
\noindent Furthermore, Lemma \ref{invariant} ensures that $\gcd(r-t,s) = \gcd(m-m,n) = n$ and $r,t \equiv m \mod(n)$.\\
~\\
\noindent At each step the algorithm uses the action of one elementary braid $\sigma_i$, so one can recover a braid $\beta_p$ sending the starting point $[r:s:t]$ to the representative $[m:n:m]$. This braid can be expressed as 
$$\beta_p = \sigma_1^{-a_{2k+1}}\sigma_3^{c_{2k+1}}\sigma_2^{b_{2k}} ~ \cdots ~ \sigma_1^{-a_3}\sigma_3^{c_3}\sigma_2^{b_2}\sigma_1^{-a_1}\sigma_3^{c_1}\sigma_2^{b_0},$$
\noindent where $b_0 = \delta_{s,0}$ and the $a_i$'s and $c_i$'s are uniquely defined by the algorithm. 
\end{proof}

\begin{rem}
The algorithm does not take into account the fact that $[m:n:m]$ is in the same orbit as $[n-m:n:n-m]$. If we really want to have only the representatives of the form $[m:n:m]$ with $m< n/2$ in the end of the algorithm, we can just add $\sigma_1\sigma_2^2\sigma_3$ as a final step in case we reached the wrong representative.
\end{rem}

Proposition~\ref{Terminator} implies that the rational projective line $\PP^2(\QQ)$ is indeed a union of the $B_4$-orbits $\Orb_{B_4}([m:n:m])$.
Theorem~\ref{decomposition_orbits} is proved.

\subsection{Connection to multidimensional continued fractions}

Let us explain in which sense our braided Euclidean algorithm 
is a Jacobi-Perron type multidimensional continued fractions (MCF) algorithm. 
For a clear description of these algorithms, see \cite{rada_periodicity_2024,karpenkov_hermites_2021}.

\begin{prop}
Consider the following subsets of $\RR_+^3$ :
$$I_0 = \{(r,s,t) ~|~ s = \min(r,s,t)\}$$
$$I_1 = \{(r,s,t)~|~t< s \leq r\} \text{ , } I_3 = \{ (r,s,t)~|~ r < s \leq t\}$$
$$I_2 = \{(r,s,t)~|~ t < r < s\} \text{ , } I_2' = \{(r,s,t)~|~ r < t < s\}$$
\noindent then the Euclid's braided algorithm is an $(\{I_0,I_1,I_2,I_2',I_3\},\{\rho_4(\sigma_1\sigma_3^{-1}),\rho_4(\sigma_1),\rho_4(\sigma_2),\rho_4(\sigma_2^{-1}),\rho_4(\sigma_3^{-1})\})$-MCF algorithm. 
\end{prop}

\begin{proof}
It is just a restatement of the steps of the braided Euclid's algorithm.
\end{proof}

\subsection{Stabilizers of the points $[m:n:m]$.}\label{StabSec}
In this subsection, we give a proof of Theorem \ref{StabConj}.

Given a point $[r:s:t]\in\PP^2(\QQ)$, its stabilizer (for the action given by $\rho$) is a subgroup of $B_4$:
$$
\mathrm{Stab}_{[r:s:t]}\subset B_4.
$$
We describe stabilizers of the representative points $[m:n:m]$ of each orbit. 
The braid Torelli group $\mathcal{B}\mathcal{I}_4 = \ker(\rho)$ is obviously a subgroup of every stabilizer. 
Recall from (\ref{GenKer}) that it is normally generated by $\tau_1^2$, $\tau_3^2$ and $\Delta$.

\begin{thm*}[\textbf{\ref{StabConj}}]
Let $n\in \NN^*$ and $0 \leq m < n$ coprime with $n$. Then
$$
\Stab_{[m:n:m]}/\mathcal{B}\mathcal{I}_4 = 
\left\lbrace
\begin{array}{l l}
\langle \tau_1\Delta,\; \sigma_2 \rangle & \text{ if } n\geq 3,\\[4pt]
\langle \tau_1\Delta, \; \sigma_2,\; \sigma_1\sigma_2^2\sigma_3 \rangle & \text{ if } n = 2,\\[4pt]
\langle \tau_1, \; \Delta,\; \sigma_2 \rangle & \text{ if } n = 1.\\
\end{array}
\right.
$$
\end{thm*}

\begin{proof} For convenience, column vectors of $\mathcal{M}_{3,1}(\QQ)$ will be denoted in line : $(u,v,w)$. Let $\beta \in \mathrm{Stab}_{[m:n:m]}$, and let $M = \rho(\beta)$,
$$ 
M = \begin{pmatrix}
		a & e & b\\
		x & f & y \\
		c & g & d\\
		\end{pmatrix}.
		$$
Since the action of $B_4$ fixes $(1,0,1)$, one has 
$$
b = 1 -a, \qquad
d = 1-c,
\qquad 
y = -x.
$$
By definition, $\beta$ stabilizes $[m:n:m]$, so either $M(m,n,m) = (m,n,m)$, or $M (m,n,m) = (-m,-n,-m)$. \\
~\\
\noindent $\bullet$ First, let us suppose that $M(m,n,m) = (m,n,m)$. Then the eigenspace associated to $1$ has dimension more than $2$, and $(0,1,0) = \frac{1}{n}(m,n,m) - m(1,0,1)$ belongs to this space. Therefore $e = g = 0$ and $f = 1$.\\
\noindent Moreover $\det(M) = 1$ implies that $c = a-1$. The matrix $M$ is then 
$$M = \begin{pmatrix}
		a & 0 & 1-a\\
		x & 1 & -x \\
		a-1 & 0 & 2-a\\
		\end{pmatrix}.$$
		
\noindent Multiplying $\beta$ by $\sigma_2^x$ to the left, we can suppose that $x = 0$. It is straightforward to check that for all $a \in \ZZ$, 
$$
\rho((\tau_1\Delta)^{a-1}) = \begin{pmatrix}
		a & 0 & 1-a\\
		0 & 1 & 0 \\
		a-1 & 0 & 2-a\\
		\end{pmatrix},
		$$

\noindent so the matrix $M$ is in the group generated by $\rho(\sigma_2)$ and $\rho(\tau_1\Delta)$.\\
~\\
\noindent $\bullet$ Now, let us suppose that $M(m,n,m) = (-m,-n,-m)$. This means 
$$
\begin{cases}
ma + ne + m(1-a) = -m\\
nf = -n\\
mc + ng + m(1-c) = -m\\
\end{cases}, \text{ so } e = g = -2\frac{m}{n}, \text{ and } f = -1.
$$
 Moreover, the matrix $M$ has $1$ and $-1$ as eigenvalues, and $\det(M) = 1$, therefore the eigenspace associated to $-1$ must have dimension $2$. Taking the trace, we get $-1 = \Tr(M) = a + f + 1-c$, so $c = a+1$. As before, we can multiply $\beta$ by $\sigma_2^{x}$ to the left to get $x = 0$, so 
$$
M = \begin{pmatrix}
a & -2m/n & 1-a\\
0 & -1 & 0 \\
a+1 & -2m/n & -a\\
\end{pmatrix}.
$$
Notice that $M(0,1,0) = (-2m/n,-1,-2m/n)$, so $\rho(\beta)([0:1:0]) = [2m:n:2m]$. The points $[0:1:0]$ and $[2m:n:2m]$ are in the same orbit only if $n = 2$ or $n = 1$. So if $n \geq 3$, this case is impossible. 

 If $n = 2$, then $m = 1$ and 
$$M = \begin{pmatrix}
a & -1 & 1-a\\
0 & -1 & 0 \\
a+1 & -1 & -a\\
\end{pmatrix} = \rho((\tau_1 \Delta)^{-a-1}\sigma_2(\sigma_1\sigma_2^2\sigma_3)).
$$

Finally, if $n = 1$, then $m = 0$, and 
$$M = \begin{pmatrix}
a & 0 & 1-a\\
0 & -1 & 0 \\
a+1 & 0 & -a\\
\end{pmatrix} = \rho(\tau_1(\tau_1\Delta)^{a+1}),
$$
\noindent and this completes the proof.
\end{proof}

\section{Topology and geometry of orbits}

In this section, we investigate the way the orbits fill the rational projective plane. In particular, we highlight the symmetries of the orbits, that are carried by an affine property of the principal orbit $\cO_1$. Finally, we outline an experimental dominance property of the principal orbit, that distinguish it from the other orbits.

\subsection{Topology of orbits}
Note that the standard embedding into the real projective plane $\PP^2(\QQ)\subset \PP^2(\RR)$
equips $\PP^2(\QQ)$ with the natural topology induced from the Euclidean norm in~$\RR^3$.
We will use this topology to study the question of density of orbits.

\begin{prop} 
\label{DenseProp}
Each orbit, except for the singleton orbit $\{[1:0:1]\}$,  is dense in $\PP^2(\QQ)$.
\end{prop}

\begin{proof}
Let us show that the orbit of $[m_0:n_0:m_0]$ is dense, for $n_0 \in \NN^*$ and coprime $m_0 ,n_0$   such that $m_0 < n_0$. 
It suffices to check that for each representative $[m:n:m]$ with $m<n/2$, there exists a sequence of points in $\Orb_{B_4}([m_0:n_0:m_0])$ converging to $[m:n:m]$. 

Consider the following sequence of points in $\PP^2(\QQ)$
$$ 
P_k := 
\left[ m + \frac{n_0+m_0-m}{n_0k+1} : 
n + \frac{n_0-n}{n_0k+1} : 
m + \frac{m_0-m}{n_0k+1} \right] 
\underset{k\rightarrow +\infty}{\longrightarrow} [m:n:m].
$$
\noindent For all $k\in \NN^*$, 
\begin{align*}
P_k &= \left[(n_0k+1)m +n_0+m_0-m : (n_0k+1)n + n_0 - n : (n_0k+1)m + m_0-m\right] \\[4pt]
&= [n_0(km+1) + m_0 : n_0(kn+1) : n_0km + m_0],
\end{align*}

\noindent with $n_0(km+1) + m_0$, $n_0(kn+1)$, $n_0km + m_0$ mutually prime because $m_0$ and $n_0$ are coprime. And 
$$\gcd(n_0(km+1)+m_0 - (n_0km +m_0), n_0(kn+1)) = \gcd(n_0,n_0(kn+1)) = n_0,$$  
\noindent therefore the point $P_k$ is in the orbit of $[m_0:n_0:m_0]$, for all $k\in \NN^*$.
\end{proof}

Proposition~\ref{DenseProp} implies that the orbits (except for the singleton orbit $\{[1 : 0 : 1]\}$) are neither closed nor open.

\subsection{Symmetries of orbits}
Identifying the plane $\QQ^2$ and the subset of points $[x:y:1] \in \PP^2(\QQ)$, we show that every orbit is stable by an action of the dihedral group~$D_4$.

\begin{prop}
The orbits are stable under the action of the dihedral group $D_4$ acting on the plane with origin placed at $(1,0) = [1:0:1]$.
\end{prop}

\begin{proof}
The action of the dihedral group $D_4$ is generated by the rotation of $\pi/2$ around the point $(1,0) = [1:0:1]$ and the symmetry with respect to the horizontal axis. Therefore, it is enough to check that for every point $[x:y:1]$, the image by the rotation $[-y+1:x-1:1]$ and the image by the symmetry $[x:-y:1]$ are in the same orbit as $[x:y:1]$.\\
~\\
Let $x = \frac{a}{b}$ and $y = \frac{c}{d}$, be irreducible fractions, and let $\delta = \gcd(b,d)$, with $b= \delta b'$, $d= \delta d'$ so that $[x:y:1] = [ad' : b'c : d'b]$ with $ad'$, $b'c$ and $d'b$ mutually prime. Then $[x:y:1]$ is in the orbit of the representative $[m:n:m]$ with $n = \gcd(a-b,c)$ and $m \equiv ad' \mod(n)$.\\
~\\
\noindent Now $[x:-y:1] = [ad':-b'c:d'b]$ so it is also in the orbit of $[m:n:m]$.\\
~\\
\noindent Similarly, $[-y+1:x-1:1] = [-bc'+bd':ad'-bd':bd']$ with $\gcd(bc',ad'-bd') = n$ and $-bc'+bd' \equiv bd' \equiv m \mod(n)$, so the point is again in the same orbit.
\end{proof}

\subsection{Affine lines in $\cO_1$}
The principal orbit $\cO_1$ has one particular property : among all orbits, it is the only that contains infinitely many straight lines of $\PP^2(\QQ)$,
see Figure~\ref{dessin}.

\begin{figure}[H]
\centering
\includegraphics[scale=0.43]{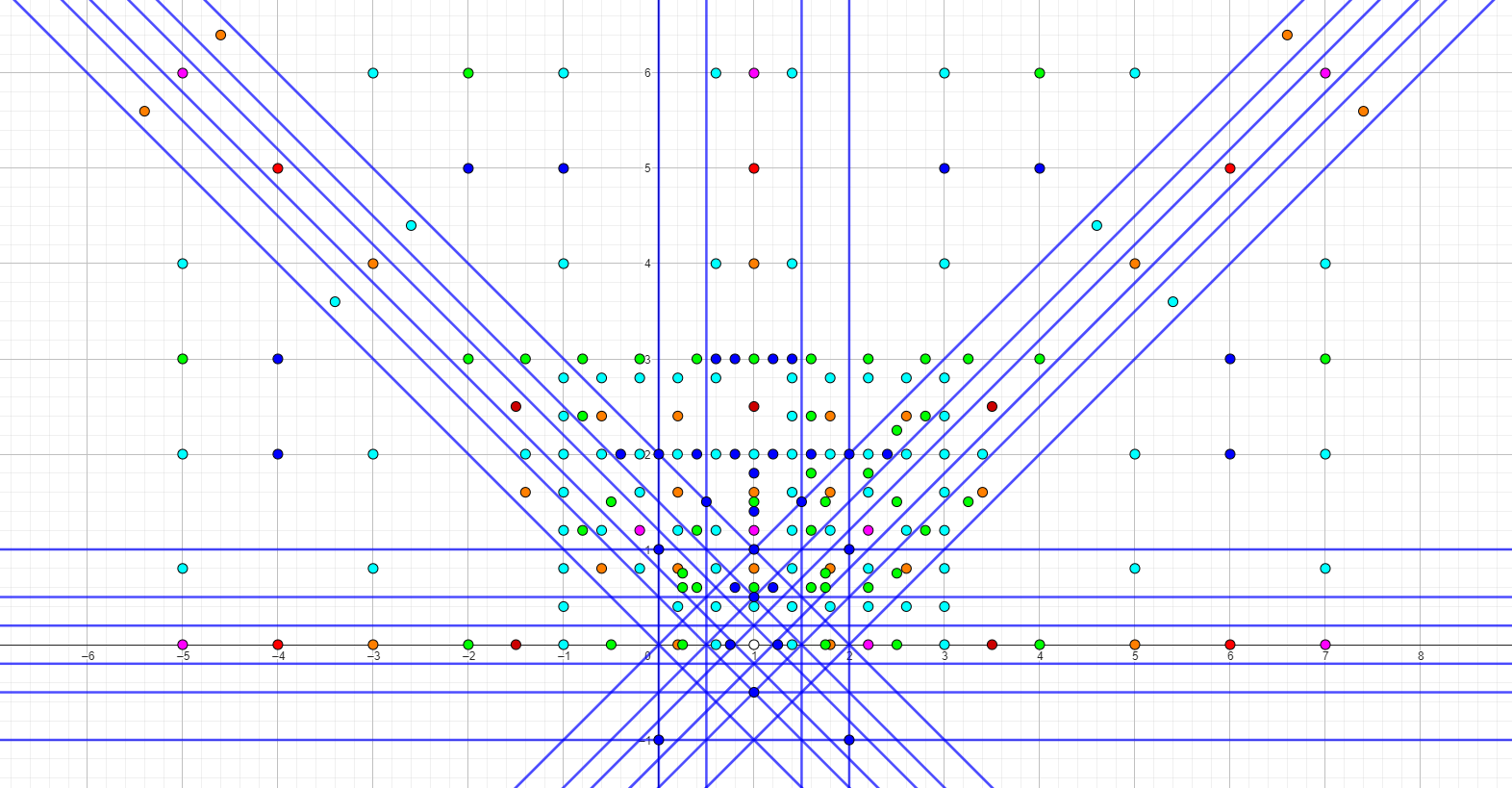}
\caption{Sketch of the orbits in  $\PP^2(\QQ)$ with the affine chart $[x:y:1]$. 
Each color represents an orbit, dark blue is for the principal orbit.}
\label{dessin}
\end{figure}

\begin{prop}
\label{LineProp}
Let $r/s \in \QQ$. The affine line of $\QQ^2$ having slope $r/s$ and passing through the point $(0,c/d)$ is entirely in $\cO_1$ if and only if 
$$cs_1 + rd_1 = \pm 1$$
\noindent with $s_1 = s/(s\wedge d)$ and $d_1 = d/(s\wedge d)$.\\
~\\
\noindent Moreover, the only vertical lines that are in $\cO_1$ are those of the form
$$D = \left\lbrace \left(\frac{r\pm 1}{r},\lambda \right),~\lambda \in \QQ \right\rbrace ~\text{ with } r\in \NN^*.$$
\end{prop}

\begin{proof}
Let $c/d \in \QQ$ in irreducible form. As $(0,c/d) \mapsto [0:c:d]$ with $\gcd(c,d) = 1$, the point $(0,c/d)$ is already in $\cO_1$. \\
\noindent Let us suppose that the line passing through $(0,c/d)$ and of slope $r/s$ is in $\cO_1$. Then for all $\lambda\in \QQ$, the point $(\lambda, \lambda r/s + c/d)$ is in $\cO_1$. Writing $\lambda = \frac{\alpha}{\beta}$ with $\alpha$ and $\beta$ coprime, we have
$$\left[\lambda:\frac{r}{s}\lambda + \frac{c}{d} : 1 \right] = [\alpha ds : r\alpha d + c\beta s: \beta ds ] = [\alpha ds_1 : r\alpha d_1 + c\beta s_1 : \beta ds_1],$$

\noindent where $d_1 = d/(d\wedge s) $ and $s_1 = s/(d\wedge s)$. As the three coordinates of this point are not necassarily mutually prime, the fact that the point is in $\cO_1$ means that
$$\gcd(\alpha ds_1,\alpha rd_1 + \beta cs_1,ds_1\beta) = \gcd(ds_1(\alpha -\beta),\alpha rd_1 + \beta cs_1).$$ 
\noindent or, because $\alpha$ and $\beta$ are coprime, 
$$\gcd(ds_1,\alpha rd_1 + \beta cs_1) = \gcd(ds_1(\alpha -\beta),\alpha rd_1 + \beta cs_1).$$ 
 
\noindent In particular, if we take $\alpha = \beta = 1$, then we get that $rd_1 + cs_1 $ divides $ds_1$.\\
\noindent Note that it means that $r\wedge c$ divides $d$ or $s$, so actually $r\wedge c = 1$.\\
~\\
\noindent Now take $\alpha = (c+d)s_1$ and $\beta = (s-r)d_1$, so that $\alpha rd_1 + \beta cs_1 =  (rd_1 + cs_1)ds_1$. We get
$$ds_1 = \pm (rd_1 + cs_1)ds_1,$$ 
\noindent so $rd_1 + cs_1 = \pm 1$. \\
\noindent For this argument to be valid, we need to check that the $\alpha$ and $\beta$ chosen above are coprime. We check that $\beta$ is coprime to $\alpha - \beta$, which is equivalent.
\begin{align*}
\gcd(\alpha-\beta,\beta) &= \gcd(rd_1 + cs_1,(s-r)d_1)\\
&= \gcd(rd_1 + cs_1,s-r) \text{ because } d_1 \wedge cs_1 = 1\\
&= \gcd(rd_1+ cs_1 , r) \text{ because } rd_1 + cs_1 \text{ divides } s\\
&= \gcd(cs_1,r) \\
& = 1,
\end{align*}
\noindent this concludes the argument.\\
~\\
\noindent Conversely, let us suppose that $rd_1+cs_1 = \pm 1$. Let $\lambda = \alpha/\beta$ in irreducible form. \\
\noindent Let us check the criteria for the point $[\alpha ds : r\alpha d + c\beta s : \beta ds] = [\alpha ds_1 : r\alpha d_1 + c\beta s_1 : \beta ds_1]$ to be in $\cO_1$. First note that the greatest common divisor of $\alpha ds_1$, $\beta ds_1$ and $r\alpha d_1 + c\beta s_1$ is $\gcd(ds_1,r\alpha d_1+ c\beta s_1)$.
One has
\begin{align*}
\gcd(\alpha ds_1 - \beta ds_1, r\alpha d_1 + c\beta s_1 ) &= \gcd((\alpha - \beta)ds_1, rd_1(\alpha - \beta) \pm \beta) \text{ ~~because } rd_1+cs_1 = \pm 1\\
&= \gcd(ds_1, rd_1(\alpha - \beta) \pm \beta) \\
&= \gcd(ds_1, r\alpha d_1 + c\beta s_1). \\
\end{align*}

\noindent Therefore $[\alpha ds_1 : r\alpha d_1 + c\beta s_1 : \beta ds_1]$ is in $\cO_1$, and the entire line is in the principal orbit.\\
~\\
\noindent For the vertical lines, let us suppose that the line $\{(a/b,\lambda)\}$ is in $\cO_1$. Then the point $[a/b:0:1] = [a:0:b]$ is in $\cO_1$, so $\gcd(a-b,0) = 1$, meaning that $a-b = \pm 1$. 

Conversely, it is straightforward to check that the line $\{((r\pm 1)/r,\lambda)\}$ is entierly in $\cO_1$. 
Proposition~\ref{LineProp} is proved.
\end{proof}

\subsection{Experimental statistics: asymptotical growth of orbits}

Let us explain why in a certain sense, the principal orbit is much bigger than the others.
Although we have no precise statement, computer experimentations clearly demonstrate this phenomenon.

\begin{defi}
Let $A$ be a subset of $\QQ_{>0}^2$. Let $d \in \NN^*$. Let $S_d$ be the set of positive rational with numerator and denominator lower that $d$ (when written in irreducible form).
We say that $A$ has rational density $\alpha$ when :
$$\frac{\sharp (A\cap S_d^2)}{\sharp S_d^2} \underset{d\rightarrow +\infty}{\longrightarrow} \alpha.$$
\end{defi}

\begin{conj}
The principal orbit $\cO_1$ has a rational density greater than $0.75$.
\end{conj}

\begin{rem}
As the natural density for positive integers, this rational density is a way to state the visual intuition that the principal orbit takes the major part of the projective plane (around $3/4$ of the plane). This definition of rational density was inspired by the work done in \cite{Lynch_2022}.

The following conjecture relies on the exact computation of the values $\sharp(\cO_1 \cap S_d^2)/\sharp S_d^2$ for large $d$ (until $d = 500$), see Figure \ref{densite}.
\begin{figure}[H]
\centering
\includegraphics[scale=0.6]{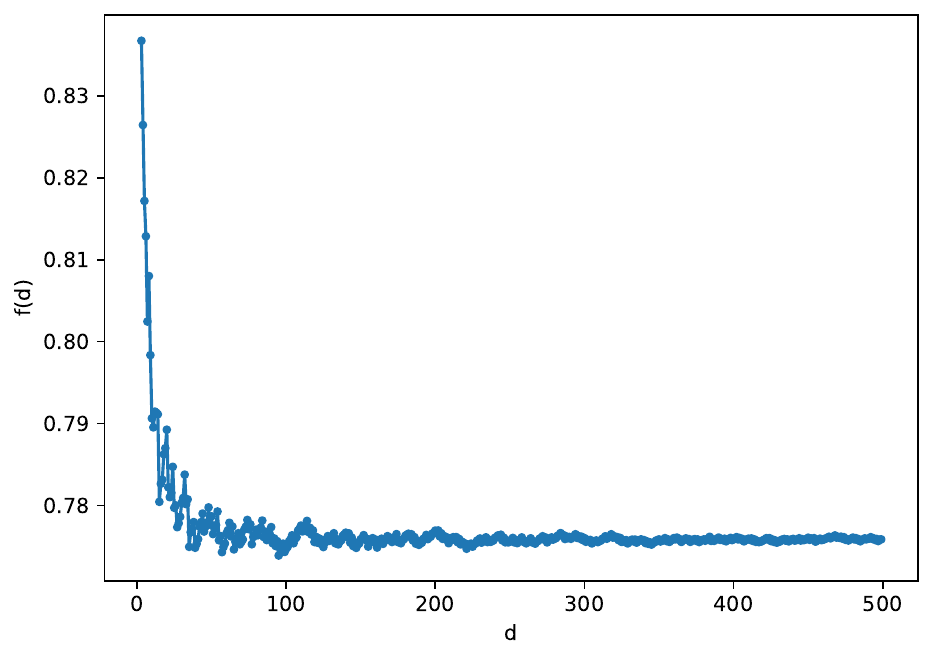}
\caption{Partial rational density of $\cO_1$ for $d = 1,\ldots, 500$.}
\label{densite}
\end{figure}
\end{rem}

\section{Quantization of $\PP^2(\QQ)$}

Now we can come back to the Burau representation $\rho_q$ where $q$ a formal variable. The study of the case $q=1$ in the previous section motivates us to focus on the principal orbit.

\subsection{Quantizing the principal orbit}

\begin{defi}
By analogy with the case where $q = 1$, let us denote by $\cO_q$ the orbit of the point $[0:1:0] \in \PP^2(\ZZ(q))$, under the action of $B_4$ via the reduced Burau representation $\rho_q$, and let us call it the quantized principal orbit.
\end{defi}

\begin{rem}
With this definition, a point $[r:s:t] \in \cO_1$ has infinitely many quantizations, depending on the braid chosen to connect $[r:s:t]$ to $[0:1:0]$.

Given a braid $\beta$ such that $\rho(\beta)([0:1:0]) = [r:s:t]$, the quantizations of the point $[r:s:t]$ are reached by $\rho_q(\beta \Stab_{[0:1:0]})([0:1:0])$. Yet, the generators computed in Proposition \ref{StabConj} have a trivial action on $[0:1:0]$ via $\rho_q$, as
$$\rho_q(\sigma_2) = \begin{pmatrix}
					1 & 0 & 0\\[4pt]
					-q & q & 1\\[4pt]
					0 & 0 & 1\\
					\end{pmatrix} \text{ , }
					\quad \rho_q(\tau_1) = \begin{pmatrix}
-q^3 & 0 & q+1\\[4pt]
0 & -q^3 & -q^2+1\\[4pt]
0 & 0 & 1\\		
\end{pmatrix} \text{ , } 
\quad \rho_q(\Delta) = \begin{pmatrix}
0 & 0 & q\\[4pt]
0 & -q^2 & 0\\[4pt]
q^3 & 0 & 0\\
\end{pmatrix}.$$
\noindent Therefore, the split of the point $[r:s:t]$ in several quantizations actually comes from the braid Torelli group $\mathcal{B}\mathcal{I}_4$.
For instance the braid $\sigma_3\tau_1^2\sigma_3^{-1}$ is quantized as 
$$\rho_q(\sigma_3\tau_1^2\sigma_3^{-1}) = \begin{pmatrix}
q^6 & q^5 + q^4 - q^2 - q  & -q^5 - q^4 + q^2 + q\\[4pt]
0 & q^4 + q^3 - q & q^6 - q^4 - q^3 + q\\[4pt]
0 & q^4 + q^3 - q - 1 & q^6 - q^4 - q^3 + q + 1\\
\end{pmatrix}.$$
\end{rem}

\begin{defi}
The quantization map is a set-valued function defined by 
$$
\begin{array}{c c c c}
\cQ : & \cO_1 & \longrightarrow & \mathcal{P}(\PP^2(\overline{\Lambda}))\\
	 & p & \longmapsto & \{ \rho_q(\beta)([0:1:0]) ~|~ \beta \text{ s.t. } \rho(\beta)([0:1:0]) = p\}\\
\end{array}.
$$
\noindent For $p\in \cO_1$, the image $\mathcal{Q}(p)$ is called the \textit{quantization} of $p$, and an element $[R:S:T]$ of $\mathcal{Q}(p)$ will be called a \textit{deformation} of $p$.\\
\end{defi}

\begin{rem} According to the previous remark, for $p\in \cO_1$, if we denote by $\beta_{p}$ the braid provided by the braided Euclidean algorithm such that $\rho(\beta_{p})([0:1:0]) = p$, we have
$$
\cQ(p) = \{\rho_q(\beta_p \gamma)([0:1:0])~|~\gamma \in \mathcal{B}\mathcal{I}_4\}.
$$
\end{rem}

\begin{ex}
Let us quantize the point $[7:18:14]$. 

(i) Thanks to the Braided Euclidean algorithm, we compute the braid $\beta = \sigma_2^2\sigma_1\sigma_3^{-3}\sigma_2^{-3}\sigma_1^3\sigma_3^{-2}$ such that $\rho(\beta)([0:1:0]) = [7:18:14]$. Then 

$$\rho_q(\beta)([0:1:0]) = [R(q):S(q):T(q)], \text{ with }$$
~\\
\noindent $R(q) = q^9 + 2q^8 + 3q^7 + 3q^6 + q^5 - q^4 - q^3 - q^2$ ,\\[4pt]
\noindent $S(q) = -q^{11} - 2q^{10} - 2q^9 + q^8 + 6q^7 + 11q^6 + 10q^5 + 4q^4 - q^3 - 4q^2 - 3q - 1$ ,\\[4pt]
\noindent $T(q) = q^8 + 3q^7 + 6q^6 + 6q^5 + 3q^4 - 2q^2 - 2q - 1.$\\
~\\
\noindent Note that these polynomials have positive and negative coefficients, and the positive (resp. the negative) parts are unimodal. \\

(ii) Let us also compute $\rho_q(\beta \sigma_3\tau_1^2\sigma_3^{-1})([0:1:0]) = [R'(q):S'(q):T'(q)]$,  \\
~\\
\noindent $R'(q) = q^{15} + 2q^{14} + 3q^{13} + 3q^{12} - 3q^{10} - 3q^9 + 3q^7 + 3q^6 + q^5 - q^4 - q^3 - q^2$ ,\\[4pt]
\noindent $S'(q) = -q^{17} - 2q^{16} - 2q^{15} + q^{14} + 7q^{13} + 13q^{12} + 11q^{11} + q^{10} - 9q^9 - 10q^8 - 2q^7 + 7q^6 + 9q^5 + 4q^4 - q^3 - 4q^2 - 3q - 1$,\\[4pt]
\noindent $T'(q) = q^{14} + 3q^{13} + 6q^{12} + 6q^{11} + 2q^{10} - 3q^9 - 5q^8 - 2q^7 + 3q^6 + 5q^5 + 3q^4 - 2q^2 - 2q - 1.$
~\\

\noindent Again, these polynomials have positive and negative parts, and each part is unimodal.
\end{ex}

\subsection{Proof of Theorem~\ref{QuantThm}}\label{QuantThmSec}
Let $r$ and $s$ be two coprime integers, with $s>0$, and such that 
$$\cQ([r:s]) = [R(q):S(q)]_q,$$ with $R(q)$, $S(q)$ two coprime polynomials.

Let us apply the braided Euclidean algorithm to the point $[r:s:0]$.\\
\noindent As the third coordinate of the point is zero, the algorithm follows exactly the steps of the usual Euclidean algorithm, so the braid provided is 
$$
\beta = \sigma_1^{-a_{2k+1}} \sigma_2^{a_{2k}} \cdots \sigma_1^{-a_3}\sigma_2^{a_2}\sigma_1^{-a_1},
$$
\noindent with the $a_i$'s being the coefficients of the continued fraction :
$$
\frac{r}{s} = a_1 + \frac{1}{a_2 + \frac{1}{\ddots + \frac{1}{a_{2k+1}}}} = [a_1,a_2,\cdots,a_{2k+1}].
$$ 
\noindent One of the deformations of $[r:s:0]$ is $\rho_q(\beta^{-1})\cdot [0:1:0]$.
$$\rho_q(\beta^{-1}) = \rho_q(\sigma_1)^{a_1}\rho_q(\sigma_2)^{-a_2} \cdots \rho_q(\sigma_1)^{a_{2k+1}} = \begin{pmatrix}
								M_q & \star \\
								0 & 1\\
								\end{pmatrix},$$
								
\noindent where $M_q$ is the image of the $3$-strands braid $\beta$ by the Burau representation of $B_3$, because of the following commutative diagramme.

\begin{center}
\begin{tikzpicture} [on grid,semithick, inner sep=2pt,bend angle=45]

\node (A) at(0,1.5) {$B_3$};
\node (B) at(3.2,1.5) {$\GL_2(\Lambda)$};
\node (C) at(0,0) {$B_4$};
\node (D) at(3.3,0) {$\GL_3(\Lambda)$};

\node (E) at(1.6,1.8) {$\rho_q$};
\node (F) at(1.6,-0.3) {$\rho_q$};

\draw (-0.8,0.7) node {$\sigma_i\mapsto \sigma_i$};
\draw (4.6,0.7) node {$M \mapsto \begin{pmatrix}
M & 0\\
0 & 1\\
\end{pmatrix}$};

\draw [thick,
         arrows={Hooks[right]->}]
    (A) edge (C)
    (B) edge (3.2,0.3);

\draw[thick,arrows={->}]
	(A) edge (B)
	(C) edge (D);

\end{tikzpicture}
\end{center}

\noindent As a $3$-strands braid, $\beta$  sends $[0:1]$ on $[r:s]$, so the second column of $M_q$ is the unique deformation of $[r:s]$ in the sense of \cite{MGO}, that is $(R(q),S(q))$, and thus $[R(q):S(q):0]$ is among the deformations of the point $[r:s:0]$.\\
~\\
The case of the other natural embedding $[r:s] \mapsto [0:r:s]$ is exactly symmetric. 

Theorem~\ref{QuantThm} is proved.

\subsection{Special specializations}

The obstruction to the unicity of quantization comes from the fact that the inclusion $\ker(\rho_q) \subset \mathcal{B}\mathcal{I}_4$ is strict (the braid Torelli group is ``too big''). We investigate specializations of $q$ at complex values for which the reverse inclusion holds.

\begin{defi}
For $z\in \CC^*$, let $\ev_z : \PGL_3(\bar{\Lambda}) \longrightarrow \PGL_3(\CC)$ be the evaluation map at $q = z$. 
\end{defi}

\begin{nota}
Let $j := \e^{\frac{2i\pi}{3}}$ be one of the third roots of unity.
\end{nota}

\begin{lemme}
The only complex values $z$ for which $\mathcal{B}\mathcal{I}_4 \subset \ker(\ev_z \rho_q)$ are $1,-1,j,j^2$.
\end{lemme}

\begin{proof}
Recall that $\mathcal{B}\mathcal{I}_4$ is normally generated by $\tau_1^2$, $\tau_3^2$ and $\Delta^2$ defined in (\ref{GenKer}), and that $\rho_q(\Delta^2) = q^4I$. \\
\noindent We can compute for all $k\in \NN$, 
$$
\rho_q(\tau_1^2)^k = \begin{pmatrix}
q^{6k} & 0 & [k]_{q^6}f_1(q)\\[4pt]
0 & q^{6k} & [k]_{q^6}f_2(q)\\[4pt]
0 & 0 & 1\\
\end{pmatrix}, \text{ with } \begin{cases}
							f_1(q) = -q^4-q^3+q+1 = -(q-1)(q+1)(q^2+q+1)\\[4pt]
							f_2(q) = q^5-q^3-q^2+1 = (q-1)^2(q+1)(q^2+q+1)\\
							\end{cases}.
$$
\noindent Likewise, 
$$
\rho_q(\tau_3^2)^k = \begin{pmatrix}
1 & 0 & 0\\[4pt]
-[k]_{q^6}g_1(q) & q^{6k} & 0\\[4pt]
-[k]_{q^6}g_2(q) & 0 & q^{6k}\\
\end{pmatrix}, \text{ with } \begin{cases}
							g_1(q) = q^6-q^4-q^3+q = q(q-1)^2(q+1)(q^2+q+1)\\[4pt]
							g_2(q) = q^6+q^5-q^3-q^2 = -q^2(q-1)(q+1)(q^2+q+1)\\
							\end{cases}.
$$

\noindent If $q = z \in \{1,-1,j,j^2\}$, then $\ev_z(\rho_q(\tau_1^2))= \ev_z(\rho_z(\tau_3^2)) = I$, so the whole normal group generated by $\tau_1^2$, $\tau_3^2$ and $\Delta^2$ is in $\ker(\ev_z \rho_q)$.\\
~\\
\noindent Conversely, if $z \notin \{1,-1,j,j^2\}$, then $f_1(z) \neq 0$ so $\ev_z(\rho_q(\tau_1^2)) \neq I$, and $\mathcal{B}\mathcal{I}_4 \nsubseteq \ker(\ev_z\rho_q)$.
\end{proof}

\begin{rem}
For $q = 1$, the statement is obvious, as $\ker(\ev_1 \rho_q) = \ker(\rho) = \mathcal{B}\mathcal{I}_4$ by definition. For $q = -1$, we have $\ker(\ev_{-1}\rho_q) = P_4$, the pure braid group with $4$ strands.

\noindent For $q = z$ a primitive third root of unity, the inclusion $\mathcal{B}\mathcal{I}_4 \subset \ker(\ev_z\rho_q)$ is strict. For instance, $\ev_z(\rho_q(\sigma_1^3)) = I$.
\end{rem}

\begin{defi}
Let $p = [r:s:t]\in \cO_1$. The $j$-deformed point associated to $p$ is defined by 
$$[r:s:t]_j := \ev_j(\rho_q(\beta))([0:1:0]),$$
\noindent for any $\beta\in B_4$ such that $\rho(\beta)([0:1:0]) = p$.
\end{defi}

\begin{ex} Let $p = [1:5:3]$. The braid $\beta = \sigma_2^{2}\sigma_1\sigma_3^{-3}$ satisfies $\rho(\beta)([0:1:0]) = p$. Then the $j$-analogue of $p$ is 
$$[1:5:3]_j = [1:-j:0].$$
\end{ex}

\begin{rem} During experimentations, we noticed that for every point $[r:s:t] \in \cO_1$ we looked at, the $j$-analogue $[R:S:T] = [r:s:t]_j$ satisfies that $R$ and $T$ are either 0 or invertible in $\ZZ[j]$ (that is $\cN(R) \leq 1 $, $\cN(T) \leq 1$) and that $\cN(S) \in \{0,1,3\}$, where $\cN$ denotes the norm of the ring of integers $\ZZ[j]$.
\end{rem}

\subsection{Minimal unimodal quantization}

The aim of this paragraph is to choose one particular deformation of a point $p \in \cO_1$ among the quantization $\mathcal{Q}(p)$.

\begin{defi}
Let $R(q) \in \ZZ[q]$. We say that $R$ is piecewise unimodal when its sequence of coefficients $(a_0,a_1,\cdots,a_n)$ is divided in subsequences of alternatively positive and negative coefficients $(+\abs{a_0},...,+\abs{a_{i_1}})$ , $(-\abs{a_{i_1+1}},\cdots, -\abs{a_{i_2}})$ , ..., each subsequence being unimodal.\\ 
\noindent Let $[R:S:T] \in \PP^2(\bar{\Lambda})$, renormalized such that $R,S$ and $T$ are polynomials in $q$. We say that $[R:S:T]$ is fully piecewise unimodal when $R$, $S$ and $T$ are piecewise unimodal.
\end{defi}

\begin{ex}
The deformation of $[3:6:4]$ obtained via the braided Euclidean algorithm is 
$$[q^5 + q^4 + q^3,
 -q^{10} - 2q^9 - 2q^8 - 2q^7 - q^6 + q^5 + 3q^4 + 4q^3 + 3q^2 + 2q + 1,
 q^3 + q^2 + q + 1].$$
 
\noindent This deformation is fully piecewise unimodal. In particular, in the second coordinate the sequence of coefficients is $(1,2,3,4,3,1,-1,-2,-2,-2,-1)$, so it has two pieces and each is unimodal.
\end{ex}

\begin{rem}
Some deformations of points in the principal orbit are not fully piecewise unimodal. For instance, in $\mathcal{Q}([21:29:11])$, there is $[R(q):S(q):T(q)]$ with
$$S(q) = q^{13} + 3q^{12} + 6q^{11} + 8q^{10} + 8q^{9} + 5q^{8} + 2q^{7} + q^{6} + 2q^{5} + 2q^{4} - q^{3} - 3q^{2} - 3q - 2,$$
\noindent which is not piecewise unimodal.
\end{rem}

\begin{conj}
Let $p \in \cO_1$. There is a unique deformation $[R:S:T]$ of $p$ in $\mathcal{Q}(p)$ such that $\deg(R)$, $\deg(S)$ and $\deg(T)$ are minimal.
\end{conj}

\begin{defi}
Assuming the conjecture above, we can define \underline{the} quantization of a point $p \in \cO_1$ to be the minimal (in degrees) deformation of $p$.\\
\noindent If $p = [r:s:t]$, we denote this minimal deformation by $[r:s:t]_q$.
\end{defi}

\begin{rem}
This definition would match with the quantization of the projective rational line embedded in $\PP^2(\QQ)$, because if $[r:s]_q = [R:S]$, then $[R:S:0]$ is minimal by unicity of the quantization of $[r:s]$. Moreover in this case, the deformation of $[r:s:0]$ is fully piecewise unimodal. 
\end{rem}

\subsection{Examples and experimentations}

To support our conjecture, we sum up some of the examples we computed.

\begin{ex}[Non unimodality] Using the braided Euclidean algorithm, we computed deformations of $[r:s:t]$ for all the triplets of nonnegative integers $(r,s,t)$ (satisfying the condition $(r-t) \wedge s = 1$ to be in $\cO_1$) bounded by $100$. In this set of examples, only 1518 (over 302172) were not fully piecewise unimodular, the first one (for the lexicographic order) occuring for the triplet $[10:67:3]$, for which the polynomials are\\
~\\
\noindent $R(q) = -q^{15} - 2q^{14} - q^{13} + q^{12} + 4q^{11} + 5q^{10} + 3q^{9} + q^{8}$,\\[4pt]
\noindent $S(q) = -q^{15} - 3q^{14} - 4q^{13} - 3q^{12} + q^{11} + 6q^{10} + 8q^{9} + 8q^{8} + 7q^{7} + 8q^{6} + 9q^{5} + 10q^{4} + 9q^{3} + 7q^{2} + 4q + 1$,\\[4pt]
\noindent $T(q) = q^{10} + q^{9} + q^{8}.$\\
~\\
\noindent For every point whose deformation using the braided Euclidean algorithm was not fully piecewise unimodal, we found an other deformation that is fully piecewise unimodal, by applying variations of the braided Euclidean algorithm. For instance, the braid $\beta = \sigma_2^{-10}\sigma_1^{-3}\sigma_3\sigma_2^2\sigma_1^{-1}$ satisfies $\rho(\beta)([0:1:0]) = [10:67:3]$, and we have $\rho_q(\beta)([0:1:0]) = [R'(q):S'(q):T'(q)]$ with \\
~\\
\noindent $R'(q) = -q^{14} - 2q^{13} - 3q^{12} - 2q^{11} - q^{10} - q^{9}$,\\[4pt]
\noindent $S'(q) = q^{15} + q^{14} - 3q^{12} - 5q^{11} - 6q^{10} - 7q^{9} - 7q^{8} - 7q^{7} - 7q^{6} - 7q^{5} - 7q^{4} - 6q^{3} - 4q^{2} - 2q - 1$,\\[4pt]
\noindent $T'(q) = -q^{16} - q^{15} - q^{14},$\\
~\\
\noindent which are piecewise unimodal. However, we can find a third deformation of $[10:67:3]$ such that the degrees of the three coordinates are together minimal. With $\beta'' = \sigma_2^{-9}\sigma_1^2\sigma_2^3\sigma_1^2\sigma_3^{-3}\sigma_2\sigma_3\tau_1^2\tau_3(\sigma_2\sigma_3)^2$, we get $\rho_q(\beta'') = [R'':S'':T'']$, with \\
~\\
\noindent $R''(q) = q^{12} + 2q^{11} + 3q^{10} + 3q^{9} + q^{8}$,\\[4pt]
\noindent $S''(q) = q^{12} + 3q^{11} + 6q^{10} + 8q^{9} + 8q^{8} + 7q^{7} + 7q^{6} + 7q^{5} + 7q^{4} + 6q^{3} + 4q^{2} + 2q + 1$,\\[4pt]
\noindent $T''(q) = q^{10} + q^{9} + q^{8}.$\\
\end{ex}

\begin{ex}[Minimality of degrees] Let us focus on three significative examples, the points $[2:1:1]$, $[3:1:5]$ and $[21:29:11]$. For these three points, we computed many different deformations in their quantizations.\\
\noindent Indeed, for a given point $p$, one can compute the braid given by the braided Euclidean algorithm, $\beta_p$, and then multiply by any element of $\mathcal{B}\mathcal{I}_4$ to get an other deformation of the point.\\
~\\
\noindent Let us denote by $ \sigma(a,b,c)$ the braid $\sigma_1^a\sigma_2^b\sigma_3^c$ with $a,b,c\in \ZZ$.\\
\noindent For each point $p$, we computed the deformations corresponding to the braids $\beta_p \gamma\tau\gamma^{-1}$, for $\tau \in \{\tau_1^2,\tau_1^2\tau_3^2\}$ and $\gamma \in \{\sigma(a_1,b_1,c_1)\sigma(a_2,b_2,c_2)\hdots \sigma(a_N,b_N,c_N) ~|~-4 \leq a_i,b_i,c_i < 4,~0 < N < 3 \}$. These parameters were chosen according to the computation power of our computer. With this method, we reach $16 514$ braids. \\
~\\
\noindent $\bullet$ Example 1 : $p = [2:1:1]$. The braided Euclidean algorithm gives $\beta_p = \sigma_1^2\sigma_3^{-1}$, and the corresponding deformation is $[q+1:1:1]$.\\
\noindent Among the $16 514$ deformations we looked at, the deformation $[q+1:1:1]$ is minimal in degrees. There was $3522$ non fully piecewise unimodal deformations, the minimal one (in degrees) being
\begin{equation*}
\begin{split}
R(q) = &-q^{20} - 4q^{19} - 8q^{18} - 10q^{17} - 7q^{16} + 7q^{14} + 10q^{13} + 11q^{12} + 12q^{11} + 11q^{10} + 5q^{9} - 3q^{8} - 7q^{7} \\
&- 6q^{6} - 2q^{5} - q^{4} - 2q^{3} - 2q^{2} - q,\\
\end{split}
\end{equation*}

$S(q) = -q^{19} - 3q^{18} - 5q^{17} - 5q^{16} - 2q^{15} + 2q^{14} + 4q^{13} + 5q^{12} + 7q^{11} + 9q^{10} + 7q^{9} - 5q^{7} - 6q^{6} - 2q^{5} - q^{3} - 2q^{2} - q,$

\begin{equation*}
\begin{split}
T(q) = &-q^{19} - 4q^{18} - 8q^{17} - 10q^{16} - 7q^{15} + 6q^{13} + 9q^{12} + 11q^{11} + 13q^{10} + 12q^{9} + 5q^{8} - 3q^{7} - 8q^{6} \\
&- 6q^{5} - 2q^{4} - q^{3} - 2q^{2} - 2q - 1.\\
\end{split}
\end{equation*}

\noindent $\bullet$ Example 2 : $p = [3:1:5]$. We applied the same protocole to this second point. Here the braided Euclidean algorithm returns $\beta_p = \sigma_1^{3}\sigma_3^{-5}$, leading to the deformation $[q^6 + q^5 + q^4:q^4:q^4 + q^3 + q^2 + q + 1]$.\\
\noindent There was $2737$ non fully piecewise unimodal deformations.\\
\noindent However, we found that the braid $\beta' = \sigma_1^3\sigma_3^{-5}\sigma_2\sigma_3\tau_1^2\tau_3^2(\sigma_2\sigma_3)^2$ provides polynomials of lower degrees than with $\beta_p$, indeed 
$$\rho_q(\beta')([0:1:0]) = [q^4 + 2q^3 + q^2 - 1: q^3 + q^2 - 1: q^3 + 2q^2 + 2q].$$
\noindent Therefore even if the braided Euclidean algorithm is efficient, it is not always the most efficient. It seams that $[q^4 + 2q^3 + q^2 - 1:q^3 + q^2 - 1:q^3 + 2q^2 + 2q]$ is the minimal deformation for $[3:1:5]$. \\
~\\
\noindent $\bullet$ Example 3 : $p = [21:29:11]$. In this example, the braided Euclidean algorithm leads to a non fully piecewise unimodal deformation of $p$. In this example, there are $2219$ non fully piecewise unimodal deformations. \\
\noindent The minimal deformation seams to be the one given by the braided Euclidean algorithm. The lowest deformation we found that is fully piecewise unimodal is\\[6pt]
\noindent $R(q) = q^{14} + 2q^{13} + 4q^{12} + 5q^{11} + 5q^{10} + 3q^{9} + q^{8} + q^{6} + 2q^{5} + q^{4} - q^{3} - 2q^{2} - q$\\[6pt]
\noindent $S(q) = q^{14} + 3q^{13} + 6q^{12} + 8q^{11} + 8q^{10} + 5q^{9} + q^{8} - q^{7} + q^{6} + 4q^{5} + 3q^{4} - q^{3} - 4q^{2} - 4q - 1$\\[6pt]
\noindent $T(q) = q^{12} + 2q^{11} + 3q^{10} + 3q^{9} + 2q^{8} - q^{6} + q^{4} + q^{3} - q$.\\
\end{ex}

\section{Discussion: open problems and future prospects} \label{XXX}

Several open problems and conjectures were spotted during experimentations with both the action by $\rho$ and $\rho_q$ on the projective rational plane.
In this section, we briefly discuss the general status of the subject.

\subsection{Non-uniqueness of quantization} 
In geometrical context, quantization usually leads to an extension of the quantized space. 
For instance, in geometric quantization the initial symplectic manifold increases its dimension by one and becomes a contact manifold.
The canonical choice of the quantized rational numbers in~\cite{MGO}
is due to the fact that quantized space (the projective line) is one-dimensional in this case.
However, even in this situation the quantization is not unique:
the second, ``left'' quantization was developed in~\cite{Bapat2022}.
Non-uniqueness was also observed in the complex case; see~\cite{ovsienko_towards_2021}.

The choice of a canonical representative of a quantized point $p\in\PP^2(\QQ)$ is a challenging problem.

\subsection{Distribution of orbits} The exact structure of the orbits of the action of $B_4$ on the rational projective plane is still mysterious. Based on large computations, we conjecture that the principal orbit takes around $3/4$ of the plane, in the sense of Proposition~\ref{DenseProp}. This experimental fact still remains to be proven. 

Moreover, what happen for the other orbits is unknown, even if we suspect that the rational density decreases when $n$ grows (where an orbit is represented by a point $[m:n:m]$). In particular, it would be interesting to study more precisely the group of symmetries for each orbit.

\subsection{Specialization at roots of unity} The evaluation of $q$ at a primitive third root of unity $j$ is a particularly simple situation thanks to the inclusion $\mathcal{B}\mathcal{I}_4 \subset \ker(\ev_j \rho_q)$. It could be worth investigating wether evualuations at other roots of unity is far from the case of the third roots. The link between the theory of $q$-rationals and the Burau representation of $B_3$ was used with success in \cite{morier-genoud_burau_2024} to study the faithfulness of specializations of this representation. For the braid group $B_4$, we hope that our quantization could perform the same types of progress. The kernels of the specializations of Burau representation at roots of unity was studied in \cite{Funar} and \cite{Dlugie2022}, following a paper of Squier \cite{Squier_1984}.

\subsection{Almost unimodality} Our computations of quantized points of the projective rational plane suggested that a large majority of the deformations are fully piecewise unimodal. We lack an explanation for this phenomenon, but it could be the sign that some deformations are better than others. The next step would be to find a combinatorial interpretation of these deformations, where unimodality would be derived from the combinatorial model.

\section*{Acknowledgements}

I am grateful to my advisor Sophie Morier-Genoud and to Valentin Ovsienko for introducing me to this topic and for all the helpful discussions we had on this problem.

This work is supported by the CDSN scolarship from the Ministère de l'Enseignement Supérieur et de la Recherche, distributed by the École Normale Supérieure de Rennes.

\newpage

\bibliographystyle{alpha}
\bibliography{biblio_B4}

\end{document}